\documentclass[12pt,reqno]{amsart}
\usepackage{amsfonts,amsmath,amscd,amssymb,amsthm}
\usepackage{enumerate,hyperref,perpage,url}
\usepackage[margin=2.2cm, a4paper]{geometry}

\theoremstyle{plain}
\newtheorem{thm}{Theorem}

\newtheorem{cor}[thm]{Corollary}
\newtheorem{lemma}[thm]{Lemma}

\newtheorem{question}[thm]{Question}
\theoremstyle{remark}

\newtheorem*{rmk}{\textbf{Remark}}

\numberwithin{equation}{section}

\MakePerPage{footnote} \allowdisplaybreaks 

\newcommand{\Av}{\mathcal A}

\newcommand{\B}{\mathcal{B}}
\newcommand{\C}{\mathbb C}

\newcommand{\dist}{\mathrm{dist}}

\newcommand{\Lip}{\mathrm{Lip}}
\newcommand\M{\mathbb M}
\newcommand\Me{\mathcal M}
\newcommand\N{\mathbb N}
\newcommand{\Or}{\mathbb{O}}

\newcommand\R{\mathbb R}
\newcommand{\s}{\mathbb S}

\newcommand\T{\mathbb T}

\newcommand{\Un}{\mathbb{U}}

\newcommand{\Vol}{\mathrm{Vol}}

\title[Small scale equidistribution of random eigenbases]{Small scale equidistribution of random eigenbases}

\author{Xiaolong Han}
\address{Department of Mathematics, The Australian National University, Canberra, ACT 2601, Australia}
\email{Xiaolong.Han@anu.edu.au}

\subjclass[2010]{58G25, 58J65, 35P20, 60B10}

\keywords{Eigenfunction equidistribution at small scales, random eigenbases, Levy concentration of measures, covering lemma}

\thanks{Research is partially supported by the Australian Research Council through Discovery Projects DP120102019 and DP150102419.}

\date{}

\begin{document}
\maketitle

\begin{abstract}
We investigate small scale equidistribution of random orthonormal bases of eigenfunctions (i.e. eigenbases) on a compact manifold $\M$. Assume that the group of isometries acts transitively on $\M$ and the multiplicity $m_\lambda$ of eigenfrequency $\lambda$ tends to infinity at least logarithmically as $\lambda\to\infty$. We prove that, with respect to the natural probability measure on the space of eigenbases, almost surely a random eigenbasis is equidistributed at small scales; furthermore, the scales depend on the growth rate of $m_\lambda$. In particular, this implies that almost surely random eigenbases on the sphere $\s^n$ ($n\ge2$) and the tori $\T^n$ ($n\ge5$) are equidistributed at polynomial scales. 
\end{abstract}

\section{Introduction}
Let $(\M,g)$ be a compact and smooth Riemannian manifold of dimension $n$ without boundary. Denote $\Delta=\Delta_g$ the (positive) Laplace-Beltrami operator and $\{e_j\}_{j=0}^\infty$ an orthonormal basis of eigenfunctions (i.e. eigenbasis) of $\Delta$ with eigenvalues $\lambda_j^2$ (counting multiplicities), i.e. $\Delta e_j=\lambda^2_je_j$, where $\lambda_j$ is called the eigenfrequency. Without loss of generality, we assume that the injectivity radius of $\M$ is greater than $1$ throughout the paper.

When the geodesic flow $G_t$ on the cosphere bundle $S^*\M$ of $\M$ is ergodic with respect to the (normalized) Liouville measure $\mu_L$, quantum ergodicity characterizes the asymptotic behaviour of the eigenbases. In particular, the quantum ergodic theorem of \v Snirel'man-Zelditch-Colin de Verdi\`ere \cite{Sn, Ze1, CdV} states that
\begin{thm}[Quantum ergodicity]\label{thm:QE} 
Assume that $G_t$ is ergodic on $S^*\M$ with respect to $\mu_L$. Then for any eigenbasis $\{e_j\}$, there is a full density subsequence of eigenfunctions $\{e_{j_k}\}\subset\{e_j\}$ such that
\begin{equation}\label{eq:QE}
\langle Ae_{j_k},e_{j_k}\big\rangle\to\mu_L(\sigma(A))\quad\text{as }k\to\infty
\end{equation}
for any pseudodifferential operator $A$ of order $0$ with principal symbol $\sigma(A)$. 
\end{thm}
Here, the density $D$ of a subsequence $J=\{j_k\}\subset\N$ is defined as
$$D(J)=\lim_{N\to\infty}\frac{\#\{j_k<N\}}{N}\quad\text{if it exists}.$$
When $D=0$ ($>0$ or $=1$), we call such subsequence a zero (positive or full) density subsequence. 

If $\langle Ae_j,e_j\big\rangle\to\mu_L(\sigma(A))$ as $j\to\infty$ for any eigenbasis $\{e_j\}$, then we say that the system of geodesic flow is quantum unique ergodic (QUE). Hassell \cite{Has} showed that classical ergodicity of $G_t$ is not sufficient to guarantee QUE. In fact, in the case of generic Bunimovich stadia, there exists zero density subsequence of eigenfunctions such that \eqref{eq:QE} is invalid.

On negatively curved manifolds (i.e. all the sectional curvatures are negative everywhere), $G_t$ has stronger properties than ergodicity e.g. central limiting, strong-mixing, and exponentially decay of correlations, etc. The quantum unique ergodicity conjecture states that QUE is valid on any compact negatively curved manifold, see Rudnick-Sarnak \cite{RS}. Restricting to Hecke eigenbases, QUE has been verified in special cases when $\M$ is arithmetic, by Lindenstrauss \cite{Lin}, Silbermann-Venkatesh \cite{SV}, Holowinsky-Soundararajan \cite{HS}, and Brooks-Lindenstrauss \cite{BrLi}.

One consequence of Theorem \ref{thm:QE} is that, the full density subsequence of eigenfunctions $\{e_{j_k}\}$ of any eigenbasis $\{e_j\}$ displays equidistribution asymptotically: Let $\Omega\subset\M$ be a Borel subset with measure-zero boundary. Then by Portmanteau theorem (c.f. Sogge \cite[Theorem 6.2.5]{So1}), we have
\begin{equation}\label{eq:einM}
\int_\Omega|e_{j_k}|^2\,d\Vol\to\frac{\Vol(\Omega)}{\Vol(\M)}\quad\text{as }k\to\infty.
\end{equation}
Here, $\Vol=\Vol_g$ is the Riemannian volume on $\M$, $\Omega$ is a fixed set independent of $\lambda$ and we say that $\Omega$ is of scale $O(1)$ (following the convention in \cite{Han}). Motivated by the applications of eigenfunction equidistribution in regions at small scales $r(\lambda)\to0$ as $\lambda\to\infty$, the author proposed the following question (\cite[Question 1.3]{Han}).

\begin{question}[Small scale equidistribution]\label{q:sseinM}
Let $\rho\in(0,1)$. Given any eigenbasis $\{e_j\}$, does there exist a full density subsequence $\{e_{j_k}\}$ such that
\begin{equation}\label{eq:sseinM}
\int_{B(x,r_{j_k})}|e_{j_k}|^2\,d\Vol=\frac{\Vol(B(x,r_{j_k}))}{\Vol(\M)}+o(r_{j_k}^n)\quad\text{as }k\to\infty
\end{equation}
for $r_{j_k}=r(\lambda_{j_k})=\lambda_{j_k}^{-\rho}$ and all $x\in\M$? Here, $B(x,r)$ is a geodesic ball in $\M$ with center $x$ and radius $r$.
\end{question} 

On arithmetic hyperbolic manifolds, such small scale equidistribution properties of eigenfunctions have already been considered in e.g. Luo-Sarnak \cite{LS} and Young \cite{Y}.

On negatively curved manifolds, small scale equidistribution was proved at logarithmic scales by the author \cite{Han} and Hezari-Rivi\`ere \cite{HR1}. Precisely,
\begin{thm}[Small scale equidistribution on negatively curved manifolds]\label{thm:ssencm}
Let $\M$ be negatively curved. Denote $r_j=(\log\lambda_j)^{-\alpha}$.
\begin{enumerate}[(i).]
\item Assume that $\alpha\in[0,1/(2n))$. Fix a point $x_0\in\M$. Then given any eigenbasis $\{e_j\}$, there exists a full density subsequence $\{e_{j_k}\}$ (depending on $x_0$) such that 
$$\int_{B(x_0,r_{j_k})}|e_{j_k}|^2\,d\Vol=\frac{\Vol(B(x_0,r_{j_k}))}{\Vol(\M)}+o(r_{j_k}^n)\quad\text{as }k\to\infty.$$
\item Assume that $\alpha\in[0,1/(3n))$. Then given any eigenbasis $\{e_j\}$, there exists a full density subsequence $\{e_{j_k}\}$ such that
\begin{equation}\label{eq:uniforminM}
c\Vol(B\big(x,r_{j_k}))\le\int_{B(x,r_{j_k})}|e_{j_k}|^2\,d\Vol\le C\Vol(B(x,r_{j_k}))\quad\text{as }k\to\infty,
\end{equation}
uniformly for all $x\in\M$, where the positive constants $c$ and $C$ depend only on $\M$.
\end{enumerate}
\end{thm}
Note that the range of $\alpha$ in (ii) was improved to $\alpha\in[0,1/(2n))$ (same as in (i)) by Hezari-Rivi\`ere \cite{HR1}. The equidistribution of eigenfunctions at logarithmic scales has since been applied to the $L^p$ norm and nodal set estimates of eigenfunctions by Hezari-Rivi\`ere \cite{HR1} and Sogge \cite{So2} and counting nodal domains of eigenfunctions by Zelditch \cite{Ze3}. See also a recent survey by Sogge \cite{So3}.

Now we switch our attention to the case when the geodesic flow is not necessarily ergodic. For example, on the $2$-$\dim$ sphere $\s^2$ where the geodesic flow is completely integrable (thus is not ergodic), the standard eigenbasis fails \eqref{eq:QE}. However, there is an infinite dimensional space of eigenbases due to high multiplicity of eigenvalues; this space carries a natural probability measure. (See \S\ref{sec:spectral} for the precise definition.) Zelditch \cite{Ze2} proved that with respect to this probability measure, almost surely a random eigenbasis $\{u_j\}$ contains a subsequence $\{u_{j_k}\}$ such that $\langle Au_{j_k},u_{j_k}\rangle\to\mu_L(\sigma(A))$ for any pseudodifferential operator $A$ of order $0$; VanderKam \cite{V} improved this result that almost surely $\langle Au_j,u_j\rangle\to\mu_L(\sigma(A))$ for a random eigenbasis $\{u_j\}$. As a consequence, even though the standard eigenbasis is not equidistributed, a typical eigenbasis is.

In this paper, we investigate small scale equidistribution of eigenbases on $\s^2$. In fact, our main theorem holds on the manifold $\M$ where
\begin{enumerate}
\item[(M1).] the group of isometries acts transitively on $\M$; 
\item[(M2).] the multiplicity $m_\lambda$ of eigenfrequency $\lambda$ satisfies
\begin{equation}\label{eq:M2}
\liminf_{\lambda\to\infty}\frac{m_\lambda}{\log\lambda}:=L_\M>0.
\end{equation}
\end{enumerate}
Then the $n$-$\dim$ sphere $\s^n$ ($n\ge2$) and the $n$-$\dim$ torus $\T^n$ ($n\ge5$) satisfies (M1) and (M2). See \S\ref{sec:ST} for the background about eigenfunctions on the spheres and the tori.

Assuming (M1) and (M2) on $\M$, we study the small scale equidistribution of the whole sequence of a random eigenbasis (as opposite to a full density subsequence as in Question \ref{q:sseinM} and Theorem \ref{thm:ssencm}). Let $\B$ be the space of eigenbases with its natural probability measure. Our main theorem states
\begin{thm}[Small scale equidistribution of random eigenbases]\label{thm:sserandom}
Assume that $\M$ satisfies (M1) and (M2). Let
$$r_j=m_{\lambda_j}^{-\alpha}\quad\text{for }0\le\alpha<\frac{1}{2n}.$$ 
Then almost surely, a random eigenbasis $\{u_j\}\in\B$ satisfies 
\begin{equation}\label{eq:sserandom}
\int_{B(x,r_j)}|u_j|^2\,d\Vol=\frac{\Vol(B(x,r_j))}{\Vol(\M)}+o(r_j^n)\quad\text{as }j\to\infty.
\end{equation}
uniformly for all $x\in\M$.
\end{thm}
Therefore, a random eigenbasis is equidistributed at small scales (depending on the multiplicity growth rate) almost surely. The present work is inspired by Burq-Lebeau \cite{BuLe1, BuLe2} and also Shiffman-Zelditch \cite{SZ}. In \cite{BuLe1, BuLe2}, Burq-Lebeau proved (among other results) that assuming similar conditions as (M1) and (M2), a random eigenbasis is almost surely bounded in $L^p$ norms, $2\le p<\infty$; In \cite{SZ}, Shiffman-Zelditch proved (among other results) that a random holomorphic section sequence is almost surely bounded in $L^p$ norms, $2\le p<\infty$, on a compact K\"ahler manifold. Their results, as well as ours, are consequences of multiplicity growth and Levy concentration of measures (see \S\ref{sec:prob}). 

\begin{rmk}
We can prove a similar result to Theorem \ref{thm:sserandom} if (M2) is only valid for a subsequence of eigenfrequencies $\{\lambda_{j_k}\}$: Let $\tilde\B$ be the space of subsequences of eigenfunctions with eigenfrequencies $\{\lambda_{j_k}\}$ endowed with a natural probability measure. Then a random subsequence of eigenfunctions $\{u_{j_k}\}\in\tilde\B$ satisfies \eqref{eq:sserandom} almost surely. Moreover, if in addition $m_{\lambda_j}=1$ for all $j\ne j_k$, then the probability measure on $\tilde\B$ is equivalent to the one on $\B$; we know that $u_j$, $j\ne j_k$, is equidistributed at all scales by Lemma \ref{lemma:spectralproj}; hence Theorem \ref{thm:sserandom} is still valid in this case.

We shall also remark that (M2) on the logarithmic growth of multiplicity is the minimal growth rate for the purpose to apply the argument in this paper. This is due to the exponential concentration rate in Levy concentration of measures. (See \S\ref{sec:proof} for more details of the proof.)
\end{rmk}

We next apply Theorem \ref{thm:sserandom} to the most well-known examples: the spheres and the tori; some of the basic facts about eigenfunctions on the spheres and the tori are gathered in \S\ref{sec:ST}. On $\s^n$ ($n\ge2$), $m_\lambda\gtrsim\lambda^{n-1}$. So by Theorem \ref{thm:sserandom},
\begin{cor}[Small scale equidistribution of random eigenbases on the spheres]\label{cor:sserandomspheres}
On $\s^n$ for $n\ge2$, let 
$$r_j=\lambda_j^{-\rho}\quad\text{for }0\le\rho<\frac{n-1}{2n}.$$
Then almost surely, a random eigenbasis $\{u_j\}\in\B$ satisfies \eqref{eq:sserandom} uniformly for all $x\in\s^n$.
\end{cor}

Note that the multiplicity growth on the spheres achieves the maximal rate that $m_\lambda\approx\lambda^{n-1}$. (See \S\ref{sec:spectral}.) Therefore, the range of the scale for equidistribution in Corollary \ref{cor:sserandomspheres} is best that one can get from Theorem \ref{thm:sserandom}. 

On $\T^n$ ($n=2,3,4$), (M2) fails; however, there are subsequences of eigenfrequencies that \eqref{eq:M2} is valid so one can derive the corresponding result about these random subsequences of eigenfunctions according to the remark below Theorem \ref{thm:sserandom}. On $\T^n$ ($n\ge5$), $m_\lambda\gtrsim\lambda^{n-2}$. So by Theorem \ref{thm:sserandom},
\begin{cor}[Small scale equidistribution of random eigenbases on the tori]\label{cor:sserandomtori}
On $\T^n$ for $n\ge5$, let 
$$r_j=\lambda_j^{-\rho}\quad\text{for }0\le\rho<\frac{n-2}{2n}.$$
Then almost surely, a random eigenbasis $\{u_j\}\in\B$ satisfies \eqref{eq:sserandom} uniformly for all $x\in\T^n$.
\end{cor}

In Hezari-Rivi\`ere \cite{HR2} and Lester-Rudnick \cite{LR}, the small scale equidistribution was proved for a \textit{full density} subsequence of \textit{any} eigenbasis. (The result in Hezari-Rivi\`ere \cite{HR2} is in fact about uniform comparability as in \eqref{eq:uniforminM}, which is a weaker version of small scale equidistribution.) Keep this difference with Corollary \ref{cor:sserandomtori} in mind. We remark that the equidistribution of toral eigenfunctions on $\T^n$ ($n\ge2$) was proved by Lester-Rudnick \cite{LR} to the scale $r=\lambda^{-\rho}$ for $\rho\in[0,1/(n-1))$, which has a smaller range when $n\ge5$ than the one in Corollary \ref{cor:sserandomtori} for random toral eigenfunctions.

\subsection*{Related results on small scale equidistribution and outline of the proof}
We shall point out the differences of small scale equidistribution of eigenbases on negatively curved manifolds (Theorem \ref{thm:ssencm}) and on manifolds satisfying (M1) and (M2) (Theorem \ref{thm:sserandom}).
In summary,
\begin{itemize}
\item On negatively curved manifolds, Theorem \ref{thm:ssencm}, in particular (ii), asserts a weaker version of small scale equidistribution (i.e. uniform comparability of volume and $L^2$-mass) of a \textit{full density} subsequence of \textit{any} eigenbasis.
\item On manifolds satisfying (M1) and (M2), Theorem \ref{thm:ssencm} asserts the equidistribution of the \textit{whole} random eigenbasis \textit{almost surely}.
\end{itemize}

In addition, the two approaches are different in nature. That is,
\begin{itemize}
\item On negatively curved manifolds, Theorem \ref{thm:ssencm} applies the correspondence of classical dynamics of geodesic flow $G_t$ and quantum dynamics of Schr\"odinger propagator $e^{it\Delta/h}$. ($h$ is the Planck parameter.) In the classical dynamics, the expnential decay of correlations of $G_t$ is essential to control the decay rate of the time-average of a symbol at small scales, see Liverani \cite{Liv}.
\item On manifolds satisfying (M1) and (M2), Theorem \ref{thm:ssencm} relies on the fact that the space of eigenbases is infinite dimensional. The proof applies crucially Levy concentration of measures that a Lipschitz function decays exponentially away from its median value on a large-dimensional sphere.
\end{itemize}

We, however, choose a similar strategy as used in \cite{Han, HR1} to prove Theorem \ref{thm:sserandom}. To be precise, we first show the small scale equidistribution for the random eigenbases at a fixed point almost surely; then we use a covering argument to pass such equidistribution property to all points on the manifold uniformly. The key difference is, due to the exponential concentration of probability measures, we can select a finer covering than the one used in \cite{Han, HR1}; this effectively provides the small scale equidistribution without conceding to uniform comparability as in \eqref{eq:uniforminM}.

\subsection*{Organization}
We gather some standard facts about eigenfunctions and probabilistic estimates in \S\ref{sec:pre}; the proof of Theorem \ref{thm:sserandom} is in \S\ref{sec:proof}.

Throughout this paper, $A\lesssim B$ ($A\gtrsim B$) means $A\le cB$ ($A\ge cB$) for some constant $c$ depending only on the manifold; $A\approx B$ means $A\lesssim B$ and $B\lesssim A$; the constants $c$ and $C$ may vary from line to line.

\section{Preliminaries}\label{sec:pre}
\subsection{The spectral decomposition}\label{sec:spectral}
On a compact manifold $\M$, we denote the spectral decomposition as
$$L^2(\M)=\oplus_{k=0}^\infty E_k,$$
where $E_k$ is the eigenspace of $\Delta$ with eigenvalue $\lambda_k^2$. For notational convenience, denote 
$$m_k:=m_{\lambda_k}=\dim(E_k)$$ 
as the multiplicity of $\lambda_k$. By H\"ormander \cite{Ho}, the Weyl asymptotics of eigenvalues states 
\begin{equation}\label{eq:Weyl}
\#\{\text{eigenvalues (counting multiplicities)}\le\lambda^2\}=c_0\lambda^n+O(\lambda^{n-1}),
\end{equation}
where $c_0$ depends only on $\M$. It follows that $\lambda_k\gtrsim k^{1/n}$. Observe also that
$$\#\{\text{eigenvalues (counting multiplicities)}\le\lambda^2\}=\sum_{l=0}^km_l,$$
where $\lambda_k\le\lambda<\lambda_{k+1}$. So $m_k\le c_1\lambda_k^{n-1}$, where $c_1$ depends only on $\M$. Furthermore, 
$$\sum_{l=1}^km_l\le\sum_{l=1}^kc_1\lambda_l^{n-1}\le c_1k\lambda_k^{n-1},$$
thus $\lambda_k\le c_2k$ by \eqref{eq:Weyl} again. Hence,
\begin{equation}\label{eq:multibd}
k^\frac1n\lesssim\lambda_k\lesssim k\quad\text{and}\quad m_k\lesssim k^{n-1}.
\end{equation}

Write $\{e_{1,k},...,e_{m_k,k}\}$ as an orthonormal basis of eigenfunctions in $E_k$. Let $\s^d_\C\subset\C^{d+1}$ be the complex unit sphere. Then any eigenfunction in $E_k$ can be written as
$$u(x)=\sum_{i=1}^{m_k}u_ie_{i,k}(x),\quad\text{where }u_i\in\C.$$
So the space of $L^2$-normalized functions in $E_k$ can be identified by $\s^{m_k-1}_\C$. Any eigenbasis can be written as
$$\{u_{i,k}\}_{k\in\N,1\le i\le m_k},$$
where $\{u_{1,k},...,u_{m_k,k}\}$ is an orthonormal basis of eigenfunctions in $E_k$. The space of eigenbases in $E_k$ can be identified as 
$$\B_k\cong\Un(m_k).$$
Here, $\Un(m_k)$ is the unitary group on $\C^{m_k}$ endowed with the probability measure as the Haar measure $\nu_k$.

The space of eigenbases $\B$ can then be identified as
$$\B\cong\times_{k=0}^\infty\Un(m_k)\quad\text{endowed with the product probability measure }\nu:=\otimes_{k=0}^\infty\nu_k.$$

\begin{rmk}
One can instead consider the randomization of eigenbases by real coefficients. That is, any $L^2$-normalized eigenfunction in $E_k$ is written as
$$u(x)=\sum_{i=1}^{m_k}u_ie_{i,k}(x),\quad\text{where }u_i\in\R.$$
The space of eigenbases $\tilde\B$ can be identified by
$$\tilde\B\cong\times_{k=0}^\infty\Or(m_k)\quad\text{with product probability measure }\tilde\nu:=\otimes_{k=0}^\infty\tilde\nu_k,$$
where $\tilde\nu_k$ is the Haar measure on the orthogonal group $\Or(m_k)$ on $\R^{m_k}$.

Then the results in Theorem \ref{thm:sserandom} and Corollaries \ref{cor:sserandomspheres} and \ref{cor:sserandomtori} are also valid, replacing $\B$ with probability measure $\nu$ by $\tilde\B$ with probability measure $\tilde\nu$. For simplicity, we only discuss $\B$ with $\nu$.
\end{rmk}

\begin{lemma}\label{lemma:spectralproj}
Assume that $\M$ satisfies (M1). Then
$$\sum_{i=1}^{m_k}|e_{i,k}(x)|^2=\frac{m_k}{\Vol(\M)}\quad\text{for all }x\in\M.$$
\end{lemma}

\begin{proof}
This is just a more general version of the theorem about zonal harmonics on the sphere (see, e.g. Sogge \cite[\S3.4]{So1}). Write the kernel of the orthogonal projection onto the space $E_k$ as
$$K_k(x,y)=\sum_{i=1}^{m_k}e_{i,k}(x)\overline{e_{i,k}(y)}.$$
It is invariant under any isometry $R$ of $\M$ so $K(Rx,Ry)=K(x,y)$; therefore $K(x,x)$ is constant on $\M$ since the isometries act transitively on $\M$. We then derive that
$$\sum_{i=1}^{m_k}|e_{i,k}(x)|^2=K_k(x,x)=\frac{m_k}{\Vol(\M)}$$
because 
$$\int_\M\sum_{i=1}^{m_k}|e_{i,k}(x)|^2=m_k.$$
\end{proof}

For any $u\in E_k$, we also observe that
\begin{equation}\label{eq:Linfty}
\|u\|_{L^\infty(\M)}\le\sqrt{K_k(x,x)}\|u\|_{L^2(\M)}=\left(\frac{m_k}{\Vol(\M)}\right)^\frac12\|u\|_{L^2(\M)}.
\end{equation}
See e.g. Sogge \cite[\S3.4]{So1}.

\subsection{Spheircal harmonics and toral eigenfunctions }\label{sec:ST}
In this subsection, we gather some standard facts about spherical harmonics (i.e. eigenfunctions on the spheres) and toral eigenfunctions on the tori.

On $\s^n$ ($n\ge2$), any spherical harmonic $u$ in $E_k$ satisfies $\Delta u=k(k+n-1)u$ so the eigenfrequency $\lambda_k=\sqrt{k(k+n-1)}\approx k$; moreover, $\dim(E_k)\approx k^{n-1}$, which achieves the maximal growth rate in the view of \eqref{eq:multibd}. See Sogge \cite[\S3.4]{So1} for more details about spherical harmonics.

On $\T^n$ ($n\ge2$), any toral eigenfunction $u$ in $E_k$ satisfies $\Delta u=\lambda_k^2u$, where $\lambda_k^2=l_1^2+\cdots+l_n^2$ for some integers $l_1,...,l_n$. When $n=2,3,4$, (M2) fails; when $n\ge5$, $\dim(E_k)\approx\lambda_k^{n-2}$, see e.g. Grosswell \cite[(9.20)]{G}.

\subsection{Probabilistic estimates}\label{sec:prob}
We now introduce the concentration of measures as the main driving force of our theorems. Let $\s^d\subset\R^{d+1}$ be the $n$-$\dim$ unit sphere endowed with its geodesic distance $\dist(\cdot,\cdot)$ and the uniform probability measure $\mu_d$. A real-valued function $F$ on $\s^d$ is said to be Lipschitz if
$$\|F\|_\Lip:=\sup_{u\ne v}\frac{|F(u)-F(v)|}{\dist(u,v)}<\infty.$$
A number $\Me(F)$ is said to be a median value of $F$ if
$$\mu_d(F\ge\Me(F))\ge\frac12\quad\text{and}\quad\mu_d(F\le\Me(F))\ge\frac12.$$
Levy concentration of measures \cite[Theorem 2.3, (1.10), and (1.12)]{Le} then asserts that a Lipschitz function on $\s^d$ is highly concentrated around its median value when $d$ is large.
\begin{thm}[Levy concentration of measures]\label{thm:Levy}
Consider a Lipschitz function $F$ on $\s^d$. Then for any $t>0$, we have
$$\mu_d(|F-\Me(F)|>t)\le\exp\left(-\frac{(d-1)t^2}{2\|F\|_\Lip^2}\right).$$
\end{thm}
We also need the probability distribution of a random eigenfunction in $E_k$. Let $\{e_{1,k},...,e_{m_k,k}\}$ be an orthonormal basis of $E_k$. Recalling the identification of the $L^2$-normalized functions in $E_k$ by $\s^{m_k-1}_\C$, we write
$$u(x)=\sum_{i=1}^{m_k}u_ie_{i,k}(x)=\langle(u_1,...,u_{m_k}),\overline{(e_{1,k}(x),...,e_{m_k,k}(x))}\rangle_{\C^{m_k}},$$
where $(u_1,...,u_{m_k})\in\s_\C^{m_k-1}$ and $(e_{1,k}(x),...,e_{m_k}(x))\in\C^{m_k}$ with length
$$|(e_{1,k}(x),...,e_{m_k,k}(x))|=\left(\frac{m_k}{\Vol(\M)}\right)^\frac12$$ 
independent of $x\in\M$ by Lemma \ref{lemma:spectralproj}. Thus, for $t\in[0,(m_k/\Vol(\M))^{1/2})$,
$$|u(x)|>t\quad\text{if and only if}\quad|\langle(u_1,...,u_{m_k}),\overline{(e_{1,k}(x),...,e_{m_k,k}(x))}\rangle_{\C^{m_k}}|>t.$$
We can identify $\s^{m_k-1}_\C$ with probability measure $P_k$ by $\s^{2m_k-1}$ with probability measure $\mu_{2m_k-1}$. We therefore have the following fact.
\begin{lemma}\label{lemma:nukt}
$$P_k(|u(x)|>t)=\begin{cases}
\left(1-\frac{\Vol(\M)t^2}{m_k}\right)^{m_k-1} & \text{if }0\le t<\left(\frac{m_k}{\Vol(\M)}\right)^\frac12,\\
0 & \text{if }t\ge\left(\frac{m_k}{\Vol(\M)}\right)^\frac12.
\end{cases}$$
for all $x\in\M$.
\end{lemma}
See e.g. \cite[\S A.1]{BuLe1} for an elementary proof.

\section{Proof of Theorem \ref{thm:sserandom}}\label{sec:proof}
In this section, we prove Theorem \ref{thm:sserandom}. We first establish \eqref{eq:sserandom} at a fixed point on the manifold; then we use a covering argument to complete the proof for all points uniformly.
\subsection{Small scale equidistribution at a fixed point}
Fix a point $x_0\in\M$. The small scale equidistribution is a consequence of Lemma \ref{lemma:nukt} and Theorem \ref{thm:Levy}; the proof in fact follows further analysis of the case $q=2$ in Burq-Lebeau \cite[\S3]{BuLe1}. 

Let $\{e_{1,k},...,e_{m_k,k}\}$ be an orthonormal basis of eigenfunctions in $E_k$. Let $r_k\ge0$ and define 
$$F_{x_0,k}(u)=\int_{B(x_0,r_k)}|u(x)|^2\,dx,$$
in which 
$$u(x)=\sum_{i=1}^{m_k}u_ie_{i,k}(x),\quad\text{where }(u_1,...,u_{m_k})\in\s^{m_k-1}_\C.$$
We know that for $q\ge0$ and a measurable function $F$ on $\s_\C^{m_k-1}$,
$$\int_{\s_\C^{m_k-1}}|F(u)|^q\,dP_k=q\int_0^\infty t^{q-1}P_k(|F(u)|>t)\,dt.$$
By Lemma \ref{lemma:nukt}, we compute the average value of $F_{x_0,k}$ as
\begin{eqnarray}
\Av(F_{x_0,k})&=&\int_{\s_\C^{m_k-1}}\int_{B(x_0,r_k)}|u(x)|^2\,dxdP_k\nonumber\\
&=&\int_{B(x_0,r_k)}\int_{\s_\C^{m_k-1}}|u(x)|^2\,dP_kdx\nonumber\\
&=&\int_{B(x_0,r_k)}2\int_0^\infty tP_k(|u(x)|>t)\,dtdx\nonumber\\
&=&\int_{B(x_0,r_k)}2\int_0^{(m_k/\Vol(\M))^{1/2}}t\left(1-\frac{\Vol(\M)t^2}{m_k}\right)^{m_k-1}\,dtdx\nonumber\\
&=&\frac{m_k}{\Vol(\M)}\int_{B(x_0,r_k)}2\int_0^1t(1-t^2)^{m_k-1}\,dtdx\nonumber\\
&=&\frac{\Vol(B(x_0,r_k))}{\Vol(\M)}.\label{eq:AvFx0}
\end{eqnarray}
This just verifies that the average value of the $L^2$-mass of a random eigenfunction $u$ in any region equals the (normalized) volume of the region, since the probability distribution of $|u(x)|$ is independent of $x\in\M$. Here, we only require that $r_k\ge0$.

To compute the median value of $F_{x_0,k}$, we evaluate its Lipschitz norm. For $u,v\in\s_\C^{m_k-1}$,
\begin{eqnarray*}
|F_{x_0,k}(u)-F_{x_0,k}(v)|&\le&\int_{B(x_0,r_k)}\big||u(x)|^2-|v(x)|^2\big|\,dx\\
&=&\int_{B(x_0,r_k)}\big|(|u(x)|-|v(x)|)(|u(x)|+|v(x)|)\big|\,dx\\
&\le&\left(\int_{B(x_0,r_k)}|u(x)-v(x)|^2\,dx\right)^\frac12\left(\int_{B(x_0,r_k)}(|u(x)|+|v(x)|)^2\,dx\right)^\frac12\\
&\le&\left\|\sum_{i=1}^{m_k}(u_i-v_i)e_{i,k}(x)\right\|_{L^2(\M)}\left(\int_{B(x_0,r_k)}(|u(x)|+|v(x)|)^2\,dx\right)^\frac12\\
&\le&c\,\dist(u,v).
\end{eqnarray*}
Therefore, $\|F_{x_0,k}\|_\Lip\le c$. By Theorem \ref{thm:Levy}, we have
\begin{eqnarray*}
|\Av(F_{x_0,k})-\Me(F_{x_0,k})|&=&\left|\|F_{x_0,k}\|_{L^1(\s_\C^{m_k-1})}-\|\Me(F_{x_0,k})\|_{L^1(\s_\C^{m_k-1})}\right|\\
&\le&\|F_{x_0,k}-\Me(F_{x_0,k})\|_{L^1(\s_\C^{m_k-1})}\\
&=&\int_0^\infty P_k(|F_{x_0,k}(u)-\Me(F_{x_0,k})|>t)\,dt\\
&\le&\int_0^\infty\exp\left(-\frac{(m_k-1)t^2}{\|F_{x_0,k}\|_\Lip^2}\right)\,dt\\
&\le&cm_k^{-\frac12}.
\end{eqnarray*}
Therefore, when
$$r_k=m_k^{-\alpha}\quad\text{for }0\le\alpha<\frac{1}{2n},$$
we have
$$|\Av(F_{x_0,k})-\Me(F_{x_0,k})|=O(m_k^{-1/2})=o(r_k^n).$$
Hence, seeing \eqref{eq:AvFx0},
\begin{equation}\label{eq:MeFx0}
\Me(F_{x_0,k})=\frac{\Vol(B(x_0,r_k))}{\Vol(\M)}+o(r_k^n).
\end{equation}
In the space of eigenbases $\B_k$ in the eigenspace $E_k$, by Theorem \ref{thm:Levy} again,
\begin{eqnarray*}
&&\nu_k\left(\{u_{i,k}\}_{i=1}^{m_k}\in\B_k:\exists1\le i\le m_k,|F_{x_0,k}(u_{i,k})-\Me(F_{x_0,k})|>t\right)\\
&\le&\sum_{i=1}^{m_k}\nu_k\left(\{u_{i,k}\}_{i=1}^{m_k}\subset E_k:|F_{x_0,k}(u_{i,k})-\Me(F_{x_0,k})|>t\right)\\
&\le&e^{-cm_kt^2}m_k.
\end{eqnarray*}
Here, we use the fact that the map $\{u_{1,k},...,u_{m_k,k}\}\to u_{i,k}$ for each $i=1,...,m_k$ sends the probability measure $\nu_k$ on $\Un(m_k)$ to $P_k$ on $\s_\C^{m_k-1}$. Then 
\begin{eqnarray*}
&&\nu\left(\{u_{i,k}\}_{k\in\N,1\le i\le m_k}\in\B:\exists k\in\N,\exists1\le i\le m_k,|F_{x_0,k}(u_{i,k})-\Me(F_{x_0,k})|>t\right)\\
&\le&\sum_{k=0}^\infty\nu_k\left(\{u_{i,k}\}_{i=1}^{m_k}\in\B_k:\exists1\le i\le m_k,|F_{x_0,k}(u_{i,k})-\Me(F_{x_0,k})|>t\right)\\
&\le&c\sum_{k=1}^\infty e^{-cm_kt^2}m_k.
\end{eqnarray*}
Here, we use the fact that $E_0$ consists of constant functions which are equidistributed at any scale.

Because $\alpha\in[0,1/(2n))$, we can find $\beta\in(\alpha n,1/2)$. Let $t_l=m_k^{-\beta}l$. Since $m_k\lesssim\lambda_k^{n-1}$ by \eqref{eq:multibd}, we compute
\begin{eqnarray*}
&&\nu\left(\{u_{i,k}\}_{k\in\N,1\le i\le m_k}\in\B:\exists k\ge1,\exists1\le i\le m_k,|F_{x_0,k}(u_{i,k})-\Me(F_{x_0,k})|>t_l\right)\\
&\le&C\sum_{k=1}^\infty\exp\left(-cm_kt_l^2\right)\lambda_k^{n-1}\\
&\le&C\sum_{k=1}^\infty\exp\left(-c(1-2\beta)l^2L_\M\log\lambda_k+(n-1)\log\lambda_k\right)
\end{eqnarray*}
by Condition (M2) that $m_k\gtrsim L_\M\log\lambda_k$. Using the fact that $k^{1/n}\lesssim\lambda_k\lesssim k$ in \eqref{eq:multibd}, 
\begin{eqnarray}
&&\nu\left(\{u_{i,k}\}_{k\in\N,1\le i\le m_k}\in\B:\exists k\ge2,\exists1\le i\le m_k,|F_{x_0,k}(u_{i,k})-\Me(F_{x_0,k})|>t_l\right)\nonumber\\
&\le&C\sum_{k=2}^\infty\exp\left(-c(1-2\beta)l^2L_M\log k+c'\log k\right)\nonumber\label{eq:nux0}\\
&\le&C\sum_{k=2}^\infty\exp\left(-cl^2\log k\right)\\
&\to&0\quad\text{as }l\to\infty\nonumber.
\end{eqnarray}
This implies that
$$\nu\left(\{u_{i,k}\}_{k\in\N,1\le i\le m_k}\in\B:\forall l\in\N,\exists k\ge2,\exists1\le i\le m_k,|F_{x_0,k}(u_{i,k})-\Me(F_{x_0,k})|>t_l\right)=0;$$
hence,
$$\nu\left(\{u_{i,k}\}_{k\in\N,1\le i\le m_k}\in\B:\exists l\in\N,\forall k\ge2,\forall1\le i\le m_k,|F_{x_0,k}(u_{i,k})-\Me(F_{x_0,k})|\le t_l\right)=1.$$
But if an eigenbasis $\{u_{i,k}\}_{k\in\N,1\le i\le m_k}\in\B$ satisfies that for some $l\in\N$,
$$|F_{x_0,k}(u_{i,k})-\Me(F_{x_0,k})|\le t_l\quad\text{for all }k\ge2,1\le i\le m_k,$$
then
$$\int_{B(x_0,r_k)}|u_{i,k}(x)|^2\,dx=\Me(F_{x_0,k})+o(r_k^n)=\frac{\Vol(B(x_0,r_k))}{\Vol(\M)}+o(r_k^n)$$
by \eqref{eq:MeFx0} and $t_l=lm_k^{-\beta}=o(m_k^{-\alpha n})=o(r_k^n)$ since $r_k=m_k^{-\alpha}$. This concludes that \eqref{eq:sserandom} is almost surely true at a fixed point $x_0\in\M$.

\subsection{Small scale equidistribution on the manifold}
To prove the small scale equidistribution uniformly for all points on the manifold, we need a covering lemma.
\begin{lemma}\label{lemma:covering}
Let $s>0$ be sufficiently small. Then there exists a family of geodesic balls that covers $\M$:
$$\bigcup_{p=1}^NB(x_p,s)\supset\M\quad\text{with }N\le c_1s^{-n},$$ 
where $c_1>0$ depends only on $\M$.
\end{lemma}

The covering lemma follows by choosing $\{B(x_p,s/2)\}_{p=1}^N$ as a maximal family of disjoint balls with radius $s/2$.

Let $\gamma>0$ be chosen later. Then Lemma \ref{lemma:covering} implies that there exists a covering $\{B(x_p,s)\}_{p=1}^N$ with
\begin{equation}\label{eq:gamma}
s=\lambda_k^{-\gamma}:=s_k\quad\text{and}\quad N\lesssim \lambda_k^{\gamma n}.
\end{equation}  
Define for $p=1,...,N$ that
$$F_{x_p,k}(u)=\int_{B(x_p,r_k)}|u(x)|^2\,dx,\quad\text{where }r_k=m_k^{-\alpha}.$$

\begin{rmk}
Notice that $\{B(x_p,r_k)\}_{p=1}^N$ is also a covering of $\M$ if $s_k\le r_k$. In fact, we shall choose $\gamma$ large enough such that $s_k=\lambda_k^{-\gamma}\ll r_k$. It is irrelevant for our purpose that the overlapping in the new covering $\{B(x_p,r_k)\}_{p=1}^N$ is not uniformly bounded. 

It is however crucial that for any $x\in\M$, there is $x_p$ such that the distance of $x$ and $x_p$ is less than $s_k\ll r_k$; so we can approximate $\Vol(B(x,r_k))$ by $\Vol(B(x_p,r_k))$ and $\int_{B(x,r_k)}|u(x)|^2\,dx$ by $\int_{B(x_p,r_k)}|u(x)|^2\,dx$ better than the corresponding step of the argument in \cite{Han, HR1}, therefore achieving uniform small scale equidistribution rather than uniform comparability of the volume and $L^2$-mass in \eqref{eq:uniforminM}.
\end{rmk}

Repeat the process in the previous subsection. We can prove that 
\begin{equation}\label{eq:MeFxp}
\Me(F_{x_p,k})=\frac{\Vol(B(x_p,r_k))}{\Vol(\M)}+o(r_k^n)\quad\text{for all }p=1,...,N;
\end{equation}
moreover, similar to \eqref{eq:nux0}, for some $\beta\in(\alpha n,1/2)$,
\begin{eqnarray*}
&&\nu\left(\{u_{i,k}\}_{k\in\N,1\le i\le m_k}\in\B:\exists k\ge2,\exists1\le i\le m_k,|F_{x_p,k}(u_{i,k})-\Me(F_{x_p,k})|>t_l\right)\\
&\le&C\sum_{k=2}^\infty\exp\left(-cl^2\log k\right).
\end{eqnarray*}
Recalling that $t_l=m_k^{-\beta}l$, by \eqref{eq:multibd},
\begin{eqnarray*}
&&\nu\left(\{u_{i,k}\}_{k\in\N,1\le i\le m_k}\in\B:\exists1\le p\le N,\exists k\ge2,\exists1\le i\le m_k,|F_{x_p,k}(u_{i,k})-\Me(F_{x_p,k})|>t_l\right)\\
&\le&C\sum_{k=2}^\infty\exp\left(-cl^2\log k\right)N\\
&\le&C\sum_{k=2}^\infty\exp\left(-cl^2\log k\right)\lambda_k^{\gamma n}\\
&\le&C\sum_{k=2}^\infty\exp\left(-cl^2\log k+\gamma n\log\lambda_k\right)\\
&\le&C\sum_{k=2}^\infty\exp\left(-cl^2\log k+c'\log k\right)\\
&\to&0\quad\text{as }l\to\infty.
\end{eqnarray*}
This implies that
\begin{eqnarray*}
&&\nu\left(\{u_{i,k}\}_{k\in\N,1\le i\le m_k}\in\B:\forall l\in\N,\exists1\le p\le N,\exists k\ge2,\exists1\le i\le m_k,|F_{x_p,k}(u_{i,k})-\Me(F_{x_p,k})|>t_l\right)\\
&=&0;
\end{eqnarray*}
hence,
\begin{eqnarray*}
&&\nu\left(\{u_{i,k}\}_{k\in\N,1\le i\le m_k}\in\B:\exists l\in\N,\forall1\le p\le N,\forall k\ge2,\forall1\le i\le m_k,|F_{x_p,k}(u_{i,k})-\Me(F_{x_p,k})|\le t_l\right)\\
&=&1.
\end{eqnarray*}
But if an eigenbasis $\{u_{i,k}\}_{k\in\N,1\le i\le m_k}\in\B$ satisfies that for some $l\in\N$,
$$|F_{x_p,k}(u_{i,k})-\Me(F_{x_p,k})|\le t_l\quad\text{for all }1\le p\le N,k\ge2,1\le i\le m_k,$$
then
\begin{equation}\label{eq:ssmrandomatxp}
\int_{B(x_p,r_k)}|u_{i,k}(x)|^2\,dx=\Me(F_{x_p,k})+o(r_k^n)=\frac{\Vol(B(x_p,r_k))}{\Vol(\M)}+o(r_k^n).
\end{equation}
by \eqref{eq:MeFxp} and $t_l=o(r_k^n)$. Note also that the above convergence is uniform for $x_p$, $p=1,...,N$, because $l$ is independent of $x_p$. Therefore, \eqref{eq:sserandom} is valid when $x_0$ is replaced by any $x_p$, $p=1,...,N$. 

Next we show that by choosing $\gamma>0$ in \eqref{eq:gamma} large enough (depending only on $\M$), \eqref{eq:ssmrandomatxp} guarantees that
$$\int_{B(z,r_k)}|u_{i,k}(x)|^2\,dx=\frac{\Vol(B(z,r_k))}{\Vol(\M)}+o(r_k^n)\quad\text{uniformly for all }z\in\M,$$
hence finishes the proof of Theorem \ref{thm:sserandom}.

Indeed, let $z\in\M$. Since $\{B(x_p,s_k)\}_{p=1}^N$ is a covering of $\M$, there exists $1\le p\le N$ such that $z\in B(x_p,s_k)$. Observe that from \eqref{eq:multibd},
$$s_k=\lambda_k^{-\gamma}\le cm_k^{-\frac{\gamma}{n-1}}=cr_k^{\frac{\gamma}{\alpha(n-1)}}<r_k,\quad\text{given that }\gamma>\frac{n-1}{2n}>\alpha(n-1).$$ 
Then we immediately derive that
\begin{eqnarray}
&&\left|\frac{\Vol(B(z,r_k))}{\Vol(\M)}-\frac{\Vol(B(x_p,r_k))}{\Vol(\M)}\right|\nonumber\\
&\le&c\Vol\left(B(x_p,r_k+s_k)\setminus B(x_p,r_k-s_k)\right)\nonumber\\
&\le&cs_kr_k^{n-1}\nonumber\\
&\le&cr_k^{\frac{\gamma}{\alpha(n-1)}}r_k^{n-1}\nonumber\\
&=&o(r_k^n).\label{eq:volumedif}
\end{eqnarray}
Using the $L^\infty$ estimate of $u_{i,k}$ in \eqref{eq:Linfty} that 
$$\|u_{i,k}\|_{L^\infty(\M)}\le cm_k^\frac12=cr_k^{-\frac{1}{2\alpha}},$$ 
we also have
\begin{eqnarray}
&&\left|\int_{B(z,r_k)}|u_{i,k}(x)|^2\,dx-\int_{B(x_p,r_k)}|u_{i,k}(x)|^2\,dx\right|\nonumber\\
&\le&\int_{B(z,r_k)\cup B(x_p,r_k)}|u_{i,k}(x)|^2\,dx-\int_{B(z,r_k)\cap B(x_p,r_k)}|u_{i,k}(x)|^2\,dx\nonumber\\
&=&\int_{(B(z,r_k)\cup B(x_p,r_k))\setminus(B(z,r_k)\cap B(x_p,r_k))}|u_{i,k}(x)|^2\,dx\nonumber\\
&\le&\int_{B(x_p,r_k+s_k)\setminus B(x_p,r_k-s_k)}|u_{i,k}(x)|^2\,dx\nonumber\\
&\le&cr_k^{-\frac1\alpha}\Vol\left(B(x_p,r_k+s_k)\setminus B(x_p,r_k-s_k)\right)\nonumber\\
&\le&cr_k^{-\frac1\alpha}s_kr_k^{n-1}\nonumber\\
&\le&cr_k^{-\frac1\alpha}r_k^{\frac{\gamma}{\alpha(n-1)}}r_k^{n-1}\nonumber\\
&=&o(r_k^n).\label{eq:integraldif}
\end{eqnarray}
given that $\gamma>2(n-1)>(1+\alpha)(n-1)$.

Combining \eqref{eq:ssmrandomatxp}, \eqref{eq:volumedif}, and \eqref{eq:integraldif}, we see that
\begin{eqnarray*}
&&\left|\int_{B(z,r_k)}|u_{i,k}(x)|^2\,dx-\frac{\Vol(B(z,r_k))}{\Vol(\M)}\right|\\
&\le&\left|\frac{\Vol(B(z,r_k))}{\Vol(\M)}-\frac{\Vol(B(x_p,r_k))}{\Vol(\M)}\right|+\left|\int_{B(z,r_k)}|u_{i,k}(x)|^2\,dx-\int_{B(x_p,r_k)}|u_{i,k}(x)|^2\,dx\right|\\
&&+\left|\int_{B(x_p,r_k)}|u_{i,k}(x)|^2\,dx-\frac{\Vol(B(x_p,r_k))}{\Vol(\M)}\right|\\
&=&o(r_k^n)\quad\text{uniformly for all }z\in\M.
\end{eqnarray*}
This completes the proof of Theorem \ref{thm:sserandom}.

\section*{Acknowledgements}
It is a pleasure to thank Melissa Tacy for all our discussions that are related to randomization and its applications to harmonic analysis. I also want to thank Andrew Hassell, Steve Lester, Gabriel Rivi\`ere, and Ze\'ev Rudnick for reading the manuscript and offering suggestions that helped to improve the presentation.

\end{document}